\documentclass[12pt]{amsart}
\usepackage{amsmath, amsthm, amscd, amsfonts, amssymb, graphicx, color}
\usepackage[bookmarksnumbered, colorlinks, plainpages]{hyperref}

\usepackage{graphicx}
\usepackage{caption}
\usepackage{pgf,tikz,pgfplots}
\usepackage{mathrsfs}
\usetikzlibrary{arrows}

\definecolor{ududff}{rgb}{0.30196078431372547,0.30196078431372547,1.}
\definecolor{xdxdff}{rgb}{0.49019607843137253,0.49019607843137253,1.}

\newcommand{\cd}{{ \rm cd}}
\newtheorem{theorem}{Theorem}[section]

\newtheorem{lemma}[theorem]{Lemma}

\newtheorem{remark}[theorem]{Remark}

\newcommand{\Irr}{{\mathrm {Irr}}}
\newcommand{\cod}{{\mathrm {cod}}}

\newcommand{\kernel}{{\mathrm {ker}}}
\newcommand{\PSL}{{\mathrm {PSL}}}

\newcommand{\SL}{{\mathrm {SL}}}

\begin{document}\title[Equivalent version of Huppert's Conjecture on the  codegrees]{Equivalent version of Huppert's Conjecture on the  codegrees}
	\author[A. Bahri, Z. Akhlaghi, B. Khosravi]{Afsane Bahri \& Zeinab Akhlaghi \& Behrooz Khosravi}
	\address{ Dept. of Pure  Math.,  Faculty  of Math. and Computer Sci. \\
		Amirkabir University of Technology (Tehran Polytechnic)\\ 424,
		Hafez Ave., Tehran 15914, Iran \newline }
	\email{afsanebahri@aut.ac.ir}
	\email{z$\_$akhlaghi@aut.ac.ir}
	\email{khosravibbb@yahoo.com}

	\thanks{}
	
	
	\begin{abstract}
		Let $G$ be a finite group, $\Irr(G)$  the set of all irreducible complex characters of $G$ and $\chi \in \Irr(G)$. Let also $\cod(\chi) = |G: \kernel\chi|/\chi(1)$ and  $ \cod(G) = \{\cod(\chi) \ | \  \chi \in \Irr(G)\} $. In this note, we show that the simple group $\PSL(2,q)$, for a prime power $q>3$, is uniquely determined by the set of its codegree. 
		
	\end{abstract}
	
	\maketitle
	\section{ Introduction}
	
Throughout this note,  $G$ is a finite group and  $\Irr(G)$ is the set of all irreducible complex characters of $G$. Let $\cd(G)$ be the set of all irreducible character degrees of G, that is, $\cd(G) = \{\chi(1) | \chi \in \Irr(G)\}$. In 1990, Huppert proposed the following conjecture:
	
	\smallskip

	{\bf Huppert's Conjecture.} Let $H$ be any finite non-abelian simple group and $G$ a finite group such that $\cd(G) = \cd(H)$. Then, $G\cong H \times A$, where $A$ is abelian.
	
	\smallskip
	Many people were devoted to the study of this problem. An  analogues of Huppert's conjecture can be proposed and studied  for any set of integers related to a finite group. For instance,  a dual version of  Huppert's conjecture  for the set of conjugacy class sizes is  considered in \cite{N, ZHTJM, ZM}. In those  papers, the conjecture verified for some families of simple groups such as $\PSL(2,q)$. 
	
In this paper, we consider a different set of integers related to a finite group $G$.  For an irreducible character $\chi$ of $G$, the codegree of  $\chi$ is defined as $\cod(\chi) = |G: \kernel\chi|/\chi(1) $. Let  $ \cod(G) = \{\cod(\chi) \ | \  \chi \in \Irr(G)\} $. This definition of codegree first appeared in     \cite{Qian},    where the authors studied a graph associated with the set     $ \cod(G) $. The term co-degree of a character had earlier been used in   \cite{chillag}   for a different quantity related to the characters.  Recently, various properties of codegree have been studied in \cite{codegree}, \cite{lewis},      \cite{new} and \cite{sak}. 
	% the authors characterized some  simple projective special linear groups having some special properties on the set of  codegrees. In fact, they assumed that   $G$ is  a group such that $|\pi(d)|=2$ for all $ 1\neq d\in \cod(G)$, then,  they  characterized all  groups $G$ with  the above property  such that $|\cod(G)|\leq 5$.  It is turned out that all groups satisfying in the mentioned property are simple projective special linear groups.
	
	In this article, we are concerned with the following conjecture, inspired by Huppert's conjecture:
	
	{\bf Conjucture:} Let $G$ be a finite group and $H$ a non-abelian simple group. If  $\cod(G)=\cod(H)$,  then $G\cong H$.

	\smallskip

	Our main result is  verifing the above   conjecture for all projective special linear groups of degree 2, and it can be the first step toward the proof of the conjecture.
	
	\smallskip
	
	\textbf{Main Theorem.} Suppose that $G$ is a finite group with ${\rm cod}(G)={\rm cod(PSL}(2,q))$, where $q=r^f>3$, for some prime $r$. Then, $G\cong {\rm PSL}(2,q)$.
	
	\smallskip

	Let $G$ be a group acting on a module $M$ over a finite field, and $r$ a prime divisor of $|G/{\bf C}_G(M)|$. If for every $\nu \in M\setminus\{0\}$, ${\bf C}_G(\nu)$ contains a Sylow $r$-subgroup of $G$ as a normal subgroup, then we say the pair $(G,M)$ satisfies $N_r$ (for more details see \cite{Nq}). We use this definition to prove the main result. If $N\unlhd G$ and $\theta\in {\rm Irr}(N)$, then  ${I}_G(\theta)$ denotes the inertia group of $\theta$ in $G$, and ${\rm Irr}(G|\theta)$ denotes the set of all irreducible constituents of $\theta^G$. Moreover, $\Irr(G|N)=\Irr(G)-\Irr(G/N)$. We also mean by $n_p$, the $p$-part of $n$, where $n$ is a natural number. For the rest of notations, we follow \cite{isaacs}.
	
	%\smallskip

%%%%%%%%%%%%%%%%%%%%%%%%%%%%%%%%%%%%
%%%%%%%%%%%%%%%%%%%%%%%%%%%%%%%%%%%%%%%%%
\section{Preliminaries}
 In this section, we collect the lemmas used throughout the paper.  

\begin{lemma}\label{size Nq} \rm(\cite[Proposition 8]{Nq}\rm)
If $(G,M)$ satisfies $N_q$, then $(|M|-1)/(|{\bf C}_M(Q)|-1)=n_q(G)$, where $n_q(G)$ is the number of Sylow $q$-subgroups of $G$ and $Q\in {\rm Syl}_q(G)$.
\end{lemma}

\begin{lemma}\label{extendAutS} \rm(\cite[Lemma 2]{cdsimple}\rm)
Let $S$ be a non-abelian finite simple group. Then, there exists $1_S\neq\chi\in {\rm Irr}(S)$ that extends to ${\rm Aut}(S)$.
\end{lemma}

\begin{lemma}\label{extendNtoG} \rm(\cite[Lemma 2.5]{extendNtoG}\rm)
Let $N$ be a minimal normal subgroup of $G$ such that $N=S_1\times \dots \times S_t$, where $S_i\cong S$ is a non-abelian simple group. If $\chi\in {\rm Irr}(S)$ extends to ${\rm Aut}(S)$, then $\chi\times \dots \times \chi\in {\rm Irr}(N)$ extends to $G$.
\end{lemma}

%\begin{lemma}\label{cod2}\rm(\cite[Lemma 3.1]{cod23}\rm)
%Let $G$ be a finite group. Then, $|{\rm cod}(G)|=2$ if and only if $G$ is an elementary abelian group.
%\end{lemma}

%\begin{lemma}\label{cod3}\rm(\cite[Lemma 3.4]{cod23}\rm)
%Let $G$ be a finite group with $|{\rm cod}(G)|=3$. Then, $G$ is a solvable group and one of the following holds:

%\item[(\textit{i})] $G$ is a $p$-group of nilpotency class $2$ with ${\rm cod}(G)=\{1,p,p^s\}$, where $s\geqslant 2$.

%\item[(\textit{ii})] $G$ is a Frobenius group with a complement of prime order $p$, $|G|$ is divisible by exactly two primes $p$ and $q$ and  ${\rm cod}(G)=\{1,p,p^s\}$, where $s\geqslant 1$.

%\end{lemma}

\begin{lemma}\label{cdsimple}\rm(\cite[Theorem C]{cdsimple}\rm)
Let $G$ be a non-abelian finite simple group. Then, $|{\rm cd}(G)|\geqslant 8$, or one of the following holds:

\item[(\textit{i})] $|{\rm cd}(G)|=4$ and $G={\rm PSL}(2,2^f)$,  $f\geqslant 2$, or

\item[(\textit{ii})] $|{\rm cd}(G)|=5$ and $G={\rm PSL}(2,p^f)$, $p\neq 2$, $p^f> 5$, or

\item[(\textit{iii})] $|{\rm cd}(G)|=6$ and $G={\rm ^2B_2}(2^{2f+1})$, $f\geqslant 1$, or $G={\rm PSL}(3,4)$, or

\item[(\textit{iv})] $|{\rm cd}(G)|=7$ and $G={\rm PSL}(3,3)$, ${\rm A}_7$, ${\rm M}_{11}$ or ${\rm J}_1$.
\end{lemma}

\begin{remark}\label{codpsl}

By \cite{white}, if $q=2^f\geqslant 4$, then ${\rm cod}({\rm PSL}(2,q))=\{1,q(q-1),q(q+1),q^{2}-1\}$, and if $q>5$ is an odd prime power, then ${\rm cod}({\rm PSL}(2,q))=\{1,q(q-1)/2,q(q+1)/2,(q^2-1)/2,q(q-\epsilon(q))\}$, where $\epsilon(q)=(-1)^{(q-1)/2}$.
\end{remark}

%\begin{lemma}\label{cod45}
%Let $G$ be a simple group. Then, the following holds:

%\item[(\textit{i})] If $|{\rm cod}(G)|=4$, then $G\cong {\rm PSL}(2,2^f)$, where $f\geqslant 2$;

%\item[(\textit{ii})] If $|{\rm cod}(G)|=5$, then $G\cong{\rm PSL}(2,p^f)$, where $p\neq 2$ and $p^f> 5$.
%\end{lemma}
%\begin{proof}
%It is straightforward by Lemma \ref{cdsimple}.
%\end{proof}

%%%%%%%%%%%%%%%%%%%%%%%%%
%\begin{remark}\label{smallresults} Let G be a finite group.
%\item[(\textit{i})] If $G$ is a group of Hughes-Thompson type in which $H_p(G)< G$ with prime $p$, then every quotient group of $G$ is either a group of Hughes-Thompson type or a $p$-group of exponent $p$.
%\end{remark}

%%%%%%%%%%%%%%%%%%%%%%%%%%%%%%%%%%%%%%%%%%%%%%%%%%%%%%%%%%%%%%%%%%%%
%%%%%%%%%%%%%%%%%%%%%%%%%%%%%%%%%%%%%%%%%%%%%%%%%%%%%
\section{\bf{Main Results}}

\begin{remark}\label{perfect}
	Let $G$ be a group such that ${\rm cod}(G)={\rm cod(PSL}(2,q))$, where $q>3$ is a prime power. Then, $G$ is a perfect group since otherwise ${\rm cod}(G/G')\subseteq {\rm cod}(G)$ contains a prime power.
\end{remark}

\begin{lemma}\label{G/N}
	Suppose that ${\rm cod}(G)={\rm cod(PSL}(2,q))$, where $q>3$ is a prime power. If $N$ is a maximal normal subgroup of $G$, then $G/N\cong {\rm PSL}(2,q)$.
\end{lemma}
\begin{proof}
	Suppose that $N$ is an arbitrary maximal normal subgroup of $G$. By Remark \ref{perfect}, $G/N$ is a non-abelian simple group.
	% If $N$ is trivial we are done, so we may assume $N$ is non-trivial.
		First of all, assume that $q=2^f$ with $f\geqslant 2$. By Remark \ref{codpsl}, $|{\rm cod}(G)|=4$. Since ${\rm cod}(G/N)\subseteq {\rm cod}(G)$ and $G/N$ is a non-abelian simple group, by Lemma \ref{cdsimple}, we conclude that $|{\rm cod}(G/N)|=4$. Moreover, Lemma \ref{cdsimple} implies that $G/N\cong {\rm PSL}(2,2^\alpha)$, where $\alpha\geqslant 2$. By comparing the elements in ${\rm cod}(G/N)$ and ${\rm cod}(G)$, we get that $\alpha=f$. Thus, $G/N\cong{\rm PSL}(2,q)$, as wanted.

	Now, let $q=p^f$, where $p$ is an odd prime such that $p^f>5$. We know that ${\rm cod}(G)=\{1,p^f(p^f-1)/2,p^f(p^f+1)/2,(p^{2f}-1)/2,p^f(p^f-\epsilon(q))\}$, where $\epsilon(q)=(-1)^{(q-1)/2}$. Since $G/N$ is a non-abelian simple group, $|{\rm cod}(G/N)|=4$ or $5$. Firstly, assume that  $|{\rm cod}(G/N)|=4$. Using Lemma \ref{cdsimple}, $G/N\cong {\rm PSL}(2,2^\alpha)$, where $\alpha\geqslant 2$. Thus, ${\rm cod}(G/N)=\{1,2^\alpha(2^\alpha-1),2^\alpha(2^\alpha+1),2^{2\alpha}-1\}$.	
	Obviously, $|(p^{2f}-1)/2|_2=|p^f(p^f-\epsilon(q))|_2$ and $2 \nmid p^f(p^f+\epsilon(q))/2$. Hence,  $p^f(p^f-\epsilon(q))=2^{\alpha}(2^\alpha-\zeta)$ and $(p^{2f}-1)/2=2^{\alpha}(2^{\alpha}+\zeta)$, where $\zeta\in \{\pm 1\}$. Therefore, $2^{\alpha}$ divides $(p^f-\epsilon(q))$ and $p^f$ divides $(2^{\alpha}-\zeta)$. Thus, $2^{\alpha}\leq p^f+1\leq 2^{\alpha}+2$. Hence, we get that either $p^f+1=2^{\alpha}$ or $p^f+1=2^{\alpha}+2$. In both cases, since $(p^{2f}-1)/2=2^{\alpha}(2^{\alpha}+\zeta)$,  we get that $2^{\alpha}=p^f-\epsilon(q)$ and $(p^f+\epsilon(q))/2=2^{\alpha}+\zeta$. Now, since $\epsilon(q) =\pm 1$, we get a contradiction.

	Therefore,  $|{\rm cod}(G/N)|=5$ and by Lemma \ref{cdsimple}, $G/N\cong {\rm PSL}(2,r^k)$, where $r$ is an odd prime with $r^k>5$. Hence, ${\rm cod}(G/N)=\{1,r^k(r^k-1)/2,r^k(r^k+1)/2,(r^{2k}-1)/2,r^k(r^k-\epsilon(r))\}$, where $\epsilon(r)=(-1)^{(r-1)/2}$. Now, assume that $r\not =p$.  Note that $r^k$ divides three elements in ${\rm cod}(G/N)$, and so $r^k$ must divide at least three elements in ${\rm cod}(G)$. Therefore, since $(p^f-1, p^f+1)=2$, we get that  $r^k$ must divide $(p^f-\epsilon(q))/2$. By the same discussion, $p^f$ divides $(r^k-\epsilon(r^k))/2$.  Hence, $r^k\leq (p^f+1)/2$ and $p^f\leq (r^k+1)/2$, implying that $2r^k-1\leq p^f\leq (r^k+1)/2$, a contradiction. Consequently, $r=p$,  yielding that $r^k=p^f$, as wanted. 
\end{proof}

{\bf Proof of the Main Theorem}
By the assumption, ${\rm cod}(G)={\rm cod(PSL}(2,q))$, where $q=r^f>3$, for some prime $r$. Let $N$ be a maximal normal subgroup of $G$. Thus, Lemma \ref{G/N} implies that $G/N\cong {\rm PSL}(2,q)$. It is remain to prove that $N=1$. On the contrary, suppose that $G$ is a counterexample with minimal order, i.e.  there exists a non-trivial maximal normal subroup $N$ of $G$ such that $G/N\cong {\PSL}(2,q)$.  We claim that $N$ is a minimal normal subgroup of $G$. Assume that there exists $1<M\unlhd G$ such that $M< N$. Note that since $\cod(G/N)\subseteq \cod(G/M)\subseteq \cod(G)$, we conclude that $\cod(G/M)=\cod(G)$.
%Let also $\overline{G}=G/M$ and $\overline{N}=N/M$.
 Hence, by the minimality of $G$, we get a contradiction.  Now, we consider the following steps:

\bigskip

\textbf{Step (1).\,} We claim that $N$ is an abelian subgroup of $G$.

\bigskip

On the contrary, suppose that $N$ is non-abelian. % Hence, $G$ contains normal subgroups $K\leq L\leq M$ such that $L/K$ is a chief factor of $G$.
Therefore, $N\cong S^{a}$, where $S$ is a non-abelian simple group, and $a$ a natural number. Lemmas \ref{extendAutS} and \ref{extendNtoG} imply that there exists $1_N\neq\varphi\in {\rm Irr}(N)$ which extends to $\psi\in {\rm Irr}(G)$. Notice that $(\ker\psi)N/N$ is a normal subgroup of $G/N$, implying that $\ker\psi\leq N$ or $(\ker\psi)N=G$. By the minimality of $N$, and the fact that $1_N\neq\varphi=\psi_N$, we get that $\ker\psi=1$ or $\ker\psi\times N=G$. If $\ker\psi=1$, then ${\rm cod}(\psi)=|G|/\psi(1)$ and since $\psi(1)\mid |N|$, $|G/N|=q(q^2-1)/{\gcd(2,q)}$ must divide ${\rm cod}(\psi)\in {\rm cod}(G)$, a contradiction. Hence, assume that  $\ker\psi\times N=G$. Note that by the minimality of $N$, $\ker\psi$ is a maximal normal subgroup of $G$, and so by Lemma \ref{G/N}, $N\cong G/\ker\psi\cong {\rm PSL}(2,q)\cong \ker\psi$. Hence,  $(q(q-1)/\gcd(2,q))^2\in\cod(G)$, a contradiction by Remark \ref{codpsl}.

%Firstly, assume that $(ker\psi)N=G$. Let also $K=ker\psi\cap N$. By the minimality of $N$, $K=N$ or $K=1$. The former case does not occur since $\psi_N=\varphi$. Hence, $K=1$, and so  $G\cong ker\psi \times N$. Since $ker\psi\cong G/N$, by Step(1), $ker\psi\cong {\rm PSL}(2,r)$, where $r>2$ is a prime power. Now, let $L$ be a maximal normal subgroup of $N$, and so similar to the proof given in Step(1), $N/L$ is isomorphic to ${\rm PSL}(2,r')$, where $r'>2$ is a prime power. Therefore, $G/L\cong {\rm PSL}(2,r)\times{\rm PSL}(2,r')$ which is a contradiction by the fact that ${\rm cod}(G/L)\subseteq {\rm cod}(G)={\rm cod(PSL}(2,q))$. Hence, $(ker\psi)N\leq N$, and so $ker\psi=N$. As a result,

\bigskip

\textbf{Step (2).\,}   Moreover, ${\bf C}_G(N)=N$ and $\varphi$ is faithful for every $\varphi \in {\rm Irr}(G|N)$.

\bigskip

By Step (1), $N$ is abelian. Note that ${\bf C}_G(N)/N\unlhd G/N$, and so ${\bf C}_G(N)=G$ or ${\bf C}_G(N)=N$. If ${\bf C}_G(N)=N$, then we are done. Hence, assume that ${\bf C}_G(N)=G$, and so $N\leq {\bf Z}(G)$. On the other hand, by Remark \ref{perfect}, $G$ is perfect.
Since $G/N$ is isomorphic to the simple group ${\rm PSL}(2,q)$ and $N\leq G'\cap {\bf Z}(G)$, 
using %the Schur multiplier of ${\rm PSL}(2,q)$ and 
the fact that $|N|$ is a prime power, $G\cong {\rm SL}(2,q)$ or $3.{\rm A}_6$. By \cite[p.8]{white} and the character table of   $3.{\rm A}_6$ (see \cite{atlas}), we get a contradiction.

Now, we show that $\varphi$ is faithful for every $\varphi \in {\rm Irr}(G|N)$. Suppose that there exists $\varphi \in {\rm Irr}(G|N)$ such that $\ker\varphi\neq 1$, and so $(\ker\varphi)N/N\unlhd G/N$. Consequently, $(\ker\varphi)N=G$ or $\ker\varphi\leq N$. If $(\ker\varphi)N=G$, then since $N$ is a normal minimal subgroup of $G$, we get that $G=\ker\varphi\times N$, a contradiction by Remark \ref{perfect}. Hence, $\ker\varphi\leq N$, implying that $\ker\varphi= N$, a contradiction.

\bigskip

\textbf{Step (3).\,}  If $1_N\neq\lambda\in \Irr(N)$, then for all $\theta\in\Irr(I_G(\lambda)|\lambda)$, we have $|I_G(\lambda)|/\theta(1)\in\cod(G)$. Moreover, $\theta(1)\mid |I_G(\lambda)/N|$ and $|N|$ divides $|G/N|$.

\bigskip
%Moreover, $N$ is an elementary abelian $p$-group.

% In addition,
%Note that by Step(3), ${\bf C}_G(N)=N$, and so the Normalizer-Centralizer Theorem implies that $G/N$ embeds in ${\rm Aut}(N)$.
Let $1_N\neq \lambda\in {\rm Irr}(N)$. By \cite[Theorem 6.11]{isaacs}, we get that for all $\theta\in\Irr(I_G(\lambda)|\lambda)$, ${\theta}^G\in \Irr(G)$. Note that $N\nleq \ker\theta$, and so by Step (2), $\ker \theta^G=1$, and so ${\rm cod}(\theta^G)=|{ I}_G(\lambda)|/\theta(1)$, as desired.

In addition, since ${\rm cod}(G/N)={\rm cod}(G)$, $|{ I}_G(\lambda)/N||N|/\theta(1)\in {\rm cod}(G/N)$. Using \cite[Theorem 6.15]{isaacs}, we get that $\theta(1)$ divides $|{ I}_G(\lambda)/N|$. Hence, $|N|$ divides the codegree of some irreducible character of $G/N$, implying that $|N|\mid |G/N|$.
\bigskip

\textbf{Step (4).\,} We claim that $|N|\neq q$.
\bigskip

On the contrary, suppose that $|N|=q$, where $q=r^f>3$, for some prime $r$. Hence, there exists $1_N\neq\lambda_0\in {\rm Irr}(N)$ such that ${ I}_G(\lambda_0)/N$ contains a Sylow $r$-subgroup of $G/N$. By Step (3), $|{ I}_G(\lambda_0)/N||N|/\theta_0(1)\in \cod(G)$, for all $\theta_0\in\Irr(I_G(\lambda_0)|\lambda_0)$. Note that $|N|=q$, which implies that $|{ I}_G(\lambda_0)/N|/\theta_0(1)\in \{q-1,q+1,(q-1)/2,(q+1)/2\}$. By the structure of proper subgroups of $G/N\cong {\PSL}(2,q)$ (see \cite{king}), we get that the only possibility for $I_G(\lambda_0)/N$ is being isomorphic to a Frobenius group of order $q(q-1)/\gcd(2,q-1)$. Therefore, $q(q-1)|N|/(\gcd(2,q-1)\theta_0(1))\in {\rm cod}(G)$. Hence, $q=|N|=\theta_0(1)$, for every ${\theta_0}\in {\rm Irr}({ I}_G(\lambda_0)|\lambda_0)$. Using \cite[Lemma 5.2, Theorem 6.2]{isaacs}, we get that $q(q-1)/\gcd(2,q-1)=|{ I}_G(\lambda_0)/N|=q^2k$, where $k$ is a natural number, a contradiction.

\bigskip
\textbf{Step (5).\,} If $q=2^f$, where $f\geqslant 2$ is a natural number, then  for every $1_N\neq\lambda\in{\rm Irr}(N)$, ${ I}_G(\lambda)/N$ is a Frobenius group of order $qd$, where $d\mid (q-1)$,.
\bigskip

Notice that by Step (1), $N$ is an elementary abelian $p$-group, where $p$ is a prime divisor of $|G/N|$.  Let $|N|=p^n$.  Step (2) and the Normalizer-Centralizer Theorem imply that  $G/N$ embeds in ${\rm Aut}(N)$, and so $n> 1$.

Firstly, suppose that $p\neq 2$. By Step (3), $|N|$ divides $|{\rm PSL}(2,2^f)|$, and so $|N|$ divides $2^{2f}-1$. We claim that $|N|\neq 2^f-1$ or $2^f+1$. If $p^n=|N|=2^f+1$, then $(p,n,f)=(3,2,3)$, and so $|N|=9$. Hence,  $G/N\cong {\rm PSL}(2,8)$. Using Step (2), ${\bf C}_G(N)=N$, and so the Normalizer-Centralizer Theorem implies that $|G/N|\mid |{\rm Aut}(N)|$, a contradiction. Also, $|N|=2^f-1$ has no solution since $n>1$. Therefore, $|N|\notin \{2^f-1, 2^f+1\}$. Let $1_N\neq\lambda\in {\rm Irr}(N)$. As we explained in Step (3), for all $\theta\in {\rm Irr}(I_G(\lambda)|\lambda)$, we have $|I_G(\lambda)|/\theta(1)\in {\rm cod}(G)$. Hence,  $|{ I}_G(\lambda)/N|/\theta(1)\in\{(2^{2f}-1)/|N|,2^f(2^f-1)/|N|,2^f(2^f+1)/|N|\}$. Now, by the structure of proper subgroups of $G/N\cong {\rm PSL}(2,2^f)$ (see \cite{king}) and the fact that $|N|\notin \{2^f-1, 2^f+1\}$, we get that the only possibility for ${ I}_G(\lambda)/N$ is being isomorphic to a Frobenius group of order $qd$, where $d$ is a divisor of $q-1$.

Now, assume that $p=2$. Let $\theta\in \Irr(I_G(\lambda)|\lambda)$ for some $1_N\neq\lambda\in\Irr(N)$. Since $N$ is a 2-group, we must have $|{ I}_G(\lambda)/N|/\theta(1)\in\{2^f(2^f-1)/|N|,2^f(2^f+1)/|N|\}$. By taking a look at the proper subgroups of $G/N\cong {\rm PSL}(2,2^f)$ (see \cite{king}), we get that for every $1_N\neq\lambda\in {\rm Irr}(N)$, ${ I}_G(\lambda)/N$ is isomorphic to the alternating group of degree 4, ${\rm PSL}(2,q_0)$, where $q=2^f$ is a power of $q_0$, the dihedral group of order $2(2^f-1)$ or $2(2^f+1)$, the cyclic group of order $2^f-1$ or $2^f+1$ or a Frobenius group of order $2^f(2^f-1)$.

$\bullet$ Suppose that ${ I}_G(\lambda)/N\cong {\rm A}_4$. Note that $|{ I}_G(\lambda)/N||N|/\theta(1)\in {\rm cod}(G)$ and $N$ is a 2-group. Hence, by Step (3),  we must have $3=2^f-1$, and so $f=2$. Hence, $G/N\cong {\rm A}_5$. Using Step (3), $|N|=2$ or $4$, and so by applying the Normalizer-Centralizer Theorem on $N$, we get a contradiction.

$\bullet$ Let ${ I}_G(\lambda)/N\cong {\rm PSL}(2,q_0)$. Since the Schur multiplier of  ${\rm PSL}(2,q_0)$ is equal to 1, by \cite[chapter 11]{isaacs}, $\lambda$ is extendible to ${ I}_G(\lambda)$. Hence, by Gallagher's Theorem \cite[Corollary 6.17]{isaacs}, ${\rm cd}(I_G(\lambda)|\lambda)=\{1,q_0,q_0-1,q_0+1\}$. Now, let $\theta\in\Irr(I_G(\lambda)|\lambda)$ such that $\theta(1)=q_0$. Thus, $q_0(q_0^{2}-1)|N|/\theta(1)=(q_0^{2}-1)|N|\in{\rm cod}(G)$,
%by letting $\theta(1)=q_0$, we get that $q_0(q_0^{2}-1)|N|/\theta_0(1)=(q_0^{2}-1)|N|\in{\rm cod}(G)$
 and so $|N|=q$, a contradiction by Step (4). 
 
 %On the other hand, by considering some character in $\Irr(I_G(\lambda)|\lambda)$ of degree 1,   we get that $|N|<q$, a contradiction.

%${\rm cod}(\psi)=q_0(q_0^{2}-1)|N|/\theta_0(1)=(q_0^{2}-1)|N|$. Since ${\rm cod}(\psi)\in {\rm cod}(G)$, we deduce that $|N|=q$. 

$\bullet$ Let ${ I}_G(\lambda)/N$ be isomorphic to the dihedral group of order $2d$, where $d=q-1$ or $q+1$. Since the Sylow subgroups of $D_{2d}$ are cyclic, by \cite[Corollaries 11.22,  11.31]{isaacs}, $\lambda$ is extendible to ${ I}_G(\lambda)$, and so $\theta(1)\in {\rm cd}(I_G(\lambda)|\lambda)=\{1,2\}$. Therefore, by Step (3), there exist $\psi_1,\psi_2\in\Irr(G)$ such that $\cod(\psi_1)=2\cod(\psi_2)$, a contradiction by Remark \ref{codpsl}.
%Now, by letting $\theta(1)=2$, we get that $|N|=q$ and in the case $\theta(1)=1$, we must have $|N|<q$, a contradiction. 

%

%Similar to the above discussion, if $\theta_0(1)=1$, then $|N|=2^{f-1}$. On the other hand, consider $\theta_0$ such that  $\theta_0(1)=2$. Then, $|N|=2^f$, a contradiction.

$\bullet$ Finally, assume that ${ I}_G(\lambda)/N$ is isomorphic to the cyclic group of order $q-1$ or $q+1$. Notice that by Step (3), for all  ${\theta}\in {\rm Irr}({ I}_G(\lambda)|\lambda)$, we have $\theta(1) \mid |{ I}_G(\lambda)/N|$ and $|{ I}_G(\lambda)/N||N|/\theta(1)\in {\rm cod}(G)$. Consequently, $|N|=q$, a contradiction by Step (4).

%On the other hand, since $N$ is a  2-group, there exists $1_N\neq\lambda_0\in {\rm Irr}(N)$ such that ${ I}_G(\lambda_0)/N$ contains a Sylow 2-subgroup of $G/N$. By the above cases, the only possibility for ${ I}_G(\lambda_0)/N$ is being isomorphic to a Frobenius group of order $q(q-1)$. Therefore, by Step (3), for every ${\theta_0}\in {\rm Irr}({ I}_G(\lambda_0)|\lambda_0)$, we get that  $q(q-1)|N|/\theta_0(1)\in {\rm cod}(G)$. Hence, $q=|N|=\theta_0(1)$, for every ${\theta_0}\in {\rm Irr}({ I}_G(\lambda_0)|\lambda_0)$. Using \cite[Lemma 5.2, Theorem 6.2]{isaacs}, we get that $q(q-1)=|{ I}_G(\lambda_0)/N|=q^2k$, where $k$ is a natural number, a contradiction.

%Therefore,  ${ I}_G(\lambda)/N$ is a Frobenius group of order $q(q-1)$, for every $1_N\neq\lambda\in{\rm Irr}(N)$, as wanted.

%\textbf{Step (6).\,} Let $q=2^f$, where $f\geqslant 2$ is a natural number. Then, $|N|=q^{2}$.

%Using Step(5), for every $1_N\neq\lambda\in{\rm Irr}(N)$, $\rm{I}_G(\lambda)/N$ contains a Sylow $q$-subgroup of $G/N\cong {\rm PSL}(2,q)$ (see Step(1)). Hence, the pair $(G/N,N)$ satisfies $N_2$. Lemma \ref{size Nq}, implies that $(|N|-1)/(|{\bf C}_N(Q)|-1)=n_2(G/N)$, where $Q\in {\rm Syl}_2(G/N)$.
%Let $(|N|,|{\bf C}_N(Q)|)=(p^n,p^\beta)$, where $\alpha, \beta$ are natural numbers and $p$ a prime dividing $|G|$. Hence, if $\alpha=s\beta$, then by simple calculations, $(p^\beta)^{s-1}+(p^\beta)^{s-2}+\dots=2^f$, implying that $p=2$. Thus, $(s,\beta)=(2,f)$, and so $n=2f$. As a result, $|N|=2^{2f}=q^{2}$, as wanted.

\bigskip

\textbf{Step (6).\,} Let $q=r^\alpha$, where $r$ is an odd prime such that $r^\alpha>5$. Then,  ${ I}_G(\lambda)/N$ is either a Frobenius group of order $qd$, where $d\mid (q-1)$ or it is a group of order $q$, for every $1_N\neq\lambda\in{\rm Irr}(N)$.

\bigskip

Recall that $N$ is an elementary abelian $p$-group, where $p$ is a prime dividing $|G/N|$. Let $|N|=p^n$, where $n$ is a natural number. Notice also that by the Normalizr-Centralizer Theorem  $n\neq 1$.

%Firstly, suppose that $p=r$.
By Step (3), if $1_N\neq\lambda\in{\rm Irr}(N)$, then for all $\theta\in {\rm Irr}({ I}_G(\lambda)|\lambda)$, we have $|{ I}_G(\lambda)/N|/\theta(1)\in\{(q^{2}-1)/2|N|,q(q-1)/2|N|,q(q+1)/2|N|,q(q-\epsilon(q))/|N|\}$, where $\epsilon(q)=(-1)^{(q-1)/2}$. In addition, by the structure of proper subgroups of $G/N\cong {\rm PSL}(2,q)$ (see \cite{king}), we have the following possibilities for ${ I}_G(\lambda)/N$:

%Firstly, suppose that there exists $1_N\neq\lambda\in{\rm Irr}(N)$ such that ${I}_G(\lambda)/N$ is isomorphic to $S_4$, ${\rm A}_4$, ${\rm A}_5$ or a Klein four-group. In the sequel, we investigate different possibilities for $q$.
$\bullet$ Suppose that ${I}_G(\lambda)/N\cong S_4$.  If $\lambda$ is extendible to $I_G(\lambda)$, then by Gallagher's  Theorem \cite[Corollary 6.17]{isaacs}, ${\rm cd}(I_G(\lambda)|\lambda)=\{1,2,3\}$. Hence, by Step (3), there are $\psi_1, \psi_2\in{\rm Irr}(G|\lambda)$ such that ${\rm cod}(\psi_1)=3{\rm cod}(\psi_2)$, a contradiction by Remark \ref{codpsl}. Hence, $\lambda$ is not extendible to $I_G(\lambda)$. Looking at the character degrees of  the Schur cover of  ${\rm S}_4$ and  \cite[chapter 11]{isaacs}, we get that ${\rm cd}(I_G(\lambda)|\lambda)=\{2,4\}$. Moreover, since $\lambda$ is not extendible to $I_G(\lambda)$, by \cite[Theorem 6.26, Corollary 6.27, Corollary 11.22]{isaacs}, the only possibility for $N$ is being a $2$-group. By Step(3), there exist $\psi_1, \psi_2\in{\rm Irr}(G|\lambda)$ such that ${\rm cod}(\psi_1)=2{\rm cod}(\psi_2)$. Consequently, by Remark \ref{codpsl}, ${\rm cod}(\psi_1)=q(q-\epsilon(q))$ and ${\rm cod}(\psi_2)=q(q-\epsilon(q))/2$. Hence, ${\rm cod}(\psi_1)=q(q-\epsilon(q))= 2^{n+2}3$. Since $r\neq 2$, $q=3$, a contradiction by the fact that $q>5$.

$\bullet$ Assume that ${ I}_G(\lambda)/N\cong {\rm A}_5$.  If $\lambda$ is extendible to $I_G(\lambda)$, then by Gallagher's Theorem \cite[Corollary 6.17]{isaacs}, ${\rm cd}(I_G(\lambda)|\lambda)=\{1,3,4,5\}$. Therefore, there exist $\psi_1, \psi_2\in{\rm Irr}(G|\lambda)$ such that ${\rm cod}(\psi_1)=3{\rm cod}(\psi_2)$, a contradiction by Remark \ref{codpsl}. Hence, $\lambda$ is not extendible to $I_G(\lambda)$.
%By the same discussion as above, we conclude that $\lambda$ is not extendible to  $I_G(\lambda)$.  
Looking at the character degrees of  the Schur cover of  ${\rm A}_5$, we get that  ${\rm cd}(I_G(\lambda)|\lambda)=\{2,4,6\}$. Thus, there exist $\psi_1, \psi_2\in{\rm Irr}(G|\lambda)$ such that ${\rm cod}(\psi_1)=3{\rm cod}(\psi_2)$,  a contradiction by Remark \ref{codpsl}.

%In addition, since the Schur multiplier of ${\rm A}_5$ is equal to 2, by the same argument, we get that $\theta(1)\in\{2,4,6\}$. Therefore, by considering different possibilities for $\theta(1)$, if $N$ is a $r$-group, then $q\in\{7,9,11,13,19,23,25,29,31,41,59,61,121\}$, and if $N$ is not a $r$-group, then $(q,|N|)\in\{(9,4),(11,4),(19,9),(31,8),(31,16),(31,32)\}$. Since $|G/N|$ divides $|{\rm Aut}(N)|$, we get that the only possibility is $q=31$ and $|N|=32$. By \cite[p.260]{perfectgroup}, we get that there is no perfect group of order  $32|{\rm PSL(2,31)}|$, a contradiction.

$\bullet$  If  ${ I}_G(\lambda)/N\cong {\rm A}_4$, then by the same discussion as the previous cases, $\lambda$ is not extendible to  $I_G(\lambda)$. Hence,  $N$ is a 2-group. By looking at the character degrees of $\SL(2,3)$, the  covering group of ${\rm A}_4$, we get that ${\rm cd}(I_G(\lambda)|\lambda)=\{2\}$. Hence, ${\rm cod}(\psi)=|{I}_G(\lambda)/N||N|/2= 2^{n+1}3$ for every $\psi\in {\rm Irr}(G|\lambda)$. Noting that $q\neq 3$ and $r\neq 2$, we get that $ 2^{n+1}3=(q^2-1)/2$. Consequently, by  the fact that $\gcd(q-1, q+1)=2$,  we get that either  $(q-1)/2=3$ or $(q+1)/2=3$. Note that $q>5$, and so the latter case does not occur. Hence, $q=7$ and $|N|=4$, a contradiction by Step (2) and the Normalizer-Centralizer Theorem.

%By the same discussion and the fact that the covering group of ${\rm A}_4$ is ${\rm SL(2,3)}$, if there exists $1_N\neq\lambda\in{\rm Irr}(N)$ such that  ${ I}_G(\lambda)/N\cong {\rm A}_4$, then $\theta(1)\in\{1,2,3\}$. Therefore, if $N$ is a $r$-group, then $q\in\{7,9,11,13,23,25\}$, and if $N$ is not a $r$-group, then $q=7$ and $|N|=4$. Using Step (2) and the Normalizer-Centralizer Theorem, we deduce that $|G/N|\mid |{\rm Aut}(N)|$, a contradiction. 
 
$\bullet$ Let ${ I}_G(\lambda)/N$ be isomorphic to the Klein four-group,. If $\lambda$ is not extendible to $I_G(\lambda)$, then ${\rm cd}(I_G(\lambda)|\lambda)=\{2\}$ and by \cite[Theorem 6.26, Corollary 6.27, Corollary 11.22]{isaacs}, $N$ is a $2$-group. Hence, there exists $d\in \cod(G)$ which is a power of $2$, a contradiction by Remark \ref{codpsl}. Therefore, $\lambda$ is extendible to $I_G(\lambda)$ and ${\rm cd}(I_G(\lambda)|\lambda)=\{1\}$.     Note that $2\neq r$ does not divide $|I_G(\lambda)/N|$. Hence, by looking at ${\rm cod(PSL}(2,q))$, we get that either $|N|=q$ or  $|{ I}_G(\lambda)/N||N|=4p^n=(q^2-1)/2$. By Step (4), the first case does not occur. The second case also implies that $(q+\zeta)/2=p^n$ and $q-\zeta=4$ for some $\zeta\in \{\pm 1\}$.  Since $q>5$, we get a contradiction.

% For the remaining cases, we note that if $N$ is a $r$-group, then there exists $1_N\neq\lambda_0\in {\rm Irr}(N)$ such that ${ I}_G(\lambda_0)/N$ contains a Sylow $r$-subgroup of $G/N$.  Hence, the only possibility for ${ I}_G(\lambda_0)/N$ is being isomorphic to a Frobenius group of order $q(q-1)$.

$\bullet$ Suppose that  ${ I}_G(\lambda)/N$ is isomorphic to the dihedral group of order $d$, where $d\in\{q-1,(q-1)/2,(q+1)/2,q+1\}$. We claim that $\lambda$ is extendible to ${ I}_G(\lambda)$. On the contrary, suppose that $\lambda$ is not extendible to ${ I}_G(\lambda)$. By \cite[Corollaries 6.27, 11.22, 11.31]{isaacs}, we get that the only possibility for $N$ is being a  $2$-group. Note that $|N|=2^n>2$, $\theta(1)\mid |{ I}_G(\lambda)/N|$ and $|{ I}_G(\lambda)/N||N|/\theta(1)\in{\cod}(G)$, for all $\theta\in\Irr(I_G(\lambda)|\lambda)$. Hence, $|N|\in \{(q+1)/2,q+1,q-1,(q-1)/2\}$. If $|N|=(q+1)/2$ is a power of 2, then 4 does not divide $d=q-1$. Thus, the Sylow 2-subgroup of ${ I}_G(\lambda)/N$ is cyclic, and so by \cite[Corollaries 11.22, 11.31]{isaacs}, $\lambda$ is extendible to ${ I}_G(\lambda)$, a contradiction. By the same discussion, we get that for all remaining possibilities for $|N|$ and $d$,  $\lambda$ is extendible to ${ I}_G(\lambda)$, a contradiction. Therefore,  $\lambda$ is extendible to ${ I}_G(\lambda)$, and so ${\rm cd}(I_G(\lambda)|\lambda)=\{1,2\}$. If $N$ is a $r$-group, then
% In the case $d=(q-1)/2$ or $(q+1)/2$, we must have $\theta(1)=1$ for all $\theta\in {\rm Irr}({ I}_G(\lambda)|\lambda)$,
$|N|=q$, a contradiction by Step (4). 
%On the other hand, by the discussion preceding this case, for $1_G\neq\psi_0\in{\rm Irr}(G|\lambda_0)$ and ${\theta_0}\in {\rm Irr}({ I}_G(\lambda_0)|\lambda_0)$, we have $q(q-1)|N|/\theta_0(1)={\rm cod}(\psi_0)\in {\rm cod}(G)$. Hence, $q=|N|=\theta_0(1)$, for every ${\theta_0}\in {\rm Irr}({ I}_G(\lambda_0)|\lambda_0)$. Using \cite[Lemma 5.2, Theorem 6.2]{isaacs}, we get that $|{ I}_G(\lambda_0)/N|=kq^2$, where $k$ is a natural number, a contradiction.
 Hence, $N$ is not a $r$-group, and so  $|N|\in \{(q+1)/2,q+1,q-1,(q-1)/2\}$. Consequently, we get that if $d=q+\xi$ and $|N|=q-\xi$, for $\xi\in\{\pm 1\}$, then  $\theta(1)=2$, for all $\theta\in {\rm Irr}({ I}_G(\lambda)|\lambda)$, a contradiction. Moreover, in the remaining possibilities for $d$ and $|N|$, we must have $\theta(1)=1$, for all $\theta\in {\rm Irr}({ I}_G(\lambda)|\lambda)$, a contradiction.

$\bullet$ Let ${ I}_G(\lambda)/N\cong {\rm PSL}(2,q_0)$, where $q$ is a power of $q_0$.  If $\lambda$ is extendible to $I_G(\lambda)$, then by Gallagher's Theorem \cite[Corollary 6.17]{isaacs}, ${\rm cd}(I_G(\lambda)|\lambda)=\{1,(q_0+\epsilon(q_0))/2,q_0-1,q_0,q_0+1\}$.
 % for every $\theta\in {\rm Irr}(I_G(\lambda)|\lambda)$.
  Hence, we get a contradiction by the same argument in Step (5). Therefore, $\lambda$ is not extendible to $I_G(\lambda)$. Using \cite[chapter 11]{isaacs} and the character degrees of  the Schur cover of  ${\rm PSL}(2,q_0)$, we get that if $q_0\neq 9$, then ${\rm cd}(I_G(\lambda)|\lambda)=\{q_0-1,q_0+1,(q_0-\epsilon(q_0))/2\}$. Moreover,  if $q_0=9$, then ${\rm cd}(I_G(\lambda)|\lambda)= \{4,8,10\}$, ${\rm cd}(I_G(\lambda)|\lambda)=\{3,6,9,15\}$  or ${\rm cd}(I_G(\lambda)|\lambda)=\{6,12\}$.
   Firstly, let $q_0\neq 9$ and $\theta\in\Irr(I_G(\lambda)|\lambda)$ such that $\theta(1)=q_0-1$. Thus, by Step (3), $q_0(q_0+1)|N|/2\in\cod(G)$.  Since $r\neq 2$ and $q$ is a power of $q_0$, we conclude that $N$ is a $r$-group, and so $(q_0+1)/2\in \{q-1,(q-1)/2,(q+1)/2,q+1\}$, a contradiction by the fact that $q>q_0$ is a power of $q_0$.
 % Firstly, let $q_0\neq 9$. Let also $\theta\in\Irr(I_G(\lambda)|\lambda)$ such that $\theta(1)=(q_0-\epsilon(q_0))/2$. Thus, by Step (3), $q_0(q_0+\epsilon(q_0))|N|\in\cod(G)$.  Since $q$ is a power of $q_0$, we conclude that $N$ is a $r$-group, and so $q_0+\epsilon(q_0)\in \{q-1,(q-1)/2,(q+1)/2,q+1\}$. By the fact that $q>q_0$ is a power of $q_0$, the only possibility is $\epsilon(q_0)=1$ and $q_0+1=(q-1)/2$. Let $q=q_0^\alpha$. Since $q_0^2-1$ divides $q-1$, $\alpha$ is even, and so $q_0=q_0^{\alpha/2}=3$. Thus, ${ I}_G(\lambda)/N\cong {\rm A}_4$, causing a contradiction as discussed. 
 Hence, $q_0=9$, and so $360|N|/\theta(1)\in {\rm cod}(G)$ for all  $\theta\in\Irr(I_G(\lambda)|\lambda)$.  If ${\rm cd}(I_G(\lambda)|\lambda)= \{4,8,10\}$ or ${\rm cd}(I_G(\lambda)|\lambda)=\{3,6,9,15\}$, then there exist  $\psi_1, \psi_2\in{\rm Irr}(G|\lambda)$ such that ${\rm cod}(\psi_1)=2.5{\rm cod}(\psi_2)$ or ${\rm cod}(\psi_1)=3 {\rm cod}(\psi_2)$, respectively, a contradiction by Remark \ref{codpsl}. Thus, ${\rm cd}(I_G(\lambda)|\lambda)=\{6,12\}$, and so there exist $\psi_1, \psi_2\in{\rm Irr}(G|\lambda)$ such that ${\rm cod}(\psi_1)=2{\rm cod}(\psi_2)$. Consequently, by Remark \ref{codpsl}, ${\rm cod}(\psi_1)=q(q-\epsilon(q))$ and ${\rm cod}(\psi_2)=q(q-\epsilon(q))/2$.  Hence, ${\rm cod}(\psi_1)=q(q-\epsilon(q))=2^3 3^2 5|N|/6$. Since $q>5$ is a power of 3, $N$ must be a 3-group, and so $q-\epsilon(q)=20$. Thus, $q=19$, a contradiction.

$\bullet$ Suppose that ${ I}_G(\lambda)/N\cong {\rm PGL}(2,q_0)$, where $q$ is an even power of $q_0$. 
% Notice that by the similar argument we have above, $N$ is a $r$-group.
 Let also $T/N$ be the subgroup of ${ I}_G(\lambda)/N$ which is isomorphic to ${\rm PSL}(2,q_0)$. Firstly, assume that $\lambda$ is extendible to $T$. Then, 
 %for every $\varphi\in {\rm Irr}(T|\lambda)$, we get that $\varphi(1)\in 
 ${\rm cd}(T|\lambda)=\{1,q_0,q_0-1,q_0+1,(q_0+\epsilon(q_0))/2\}$, where $\epsilon(q_0)=(-1)^{(q_0-1)/2}$. Now, choose $\varphi\in {\rm Irr}(T|\lambda)$ such that $\varphi(1)=q_0+1$. Suppose that $\varphi$ is extendible to ${ I}_G(\lambda)$. Therefore, for some $\theta\in {\rm Irr}({ I}_G(\lambda)|\varphi)$, we have $\theta(1)=q_0+1$. Thus, by Step (3), $q_0({q_0}^2-1)|N|/\theta(1)\in {\rm cod}(G)$.  Since $q$ is a power of $q_0$, we get that $N$ is a $r$-group and $q_0-1\in\{q+1,q-1,(q+1)/2,(q-1)/2\}$, a contradiction by the fact that $q>q_0$ is a power of $q_0$.
 % Now, choose $\varphi\in {\rm Irr}(T|\lambda)$ such that $\varphi(1)=(q_0+\epsilon(q_0))/2$. Suppose that $\varphi$ is extendible to ${ I}_G(\lambda)$. Therefore, for some $\theta\in {\rm Irr}({ I}_G(\lambda)|\varphi)$, we have $\theta(1)=(q_0+\epsilon(q_0))/2$. Thus, by Step (3), $q_0({q_0}^2-1)|N|/\theta(1)\in {\rm cod}(G)$.  Since $q$ is a power of $q_0$, we get that $N$ is a $r$-group and $2(q_0-\epsilon(q_0))\in\{q+1,q-1,(q+1)/2,(q-1)/2\}$. Recalling that $q$ is a power of $q_0$, we conclude that if $\epsilon(q_0)=1$, then  $2(q_0-1)=(q-1)/2$ and if $\epsilon(q_0)=-1$, then $2(q_0+1)=q-1$ or $2(q_0+1)=(q-1)/2$. Let $q=q_0^\alpha$. In the first case, $4=1+q_0+\dots+{q_0}^{\alpha-1}$. Hence, $q_0=3$, and so ${ I}_G(\lambda)/N\cong S_4$, a contradiction as we explained. For the latter possibilities, we get that $q_0^2-1$ divides $q-1$, and so $\alpha$ is even. If $2(q_0+1)=q-1$, then $q_0^{2}-1=q-1$, and so $q_0-1=2$, a contradiction. Also, in the case $2(q_0+1)=(q-1)/2$, we conclude that $q_0=q_0^{\alpha/2}=5$, and so $q=25$.  Now, since $N$ is a $r$-group and $n>1$, $|N|=25$, a contradiction by Step (2) and the Normalizer-Centralizer Theorem. 
Therefore, $\varphi$ is not extendible to ${I}_G(\lambda)$, and so $\varphi^{I_G(\lambda)}\in {\rm Irr}(I_G(\lambda))$. Setting $\theta=\varphi^{I_G(\lambda)}$, we get that $\theta(1)=2\varphi(1)$. Hence, $\theta(1)=2(q_0+1)$. Now, by the similar argument as above, we get that $(q_0-1)/2\in\{q+1,q-1,(q+1)/2,(q-1)/2\}$, a contradiction by the fact that $q>q_0$ is a power of $q_0$. 

Consequently, $\lambda$ is not extendible to $T$, and by the same discussion as the previous case, if $q_0\neq 9$, then ${\rm cd}(T|\lambda)=\{q_0-1,q_0+1,(q_0-\epsilon(q_0))/2\}$. Moreover, if $q_0=9$, then ${\rm cd}(T|\lambda)= \{4,8,10\}$, ${\rm cd}(T|\lambda)=\{3,6,9,15\}$  or ${\rm cd}(T|\lambda)=\{6,12\}$. In the case that $q_0\neq 9$, choose $\varphi$ such that $\varphi(1)=q_0-1$. By the above discussion, for some $\theta\in {\rm Irr}(I_G(\lambda)|\lambda)$, we conclude that if $\varphi$ is extendible to  ${ I}_G(\lambda)$, then $\theta(1)=q_0-1$, and if $\varphi$ is not extendible to  ${ I}_G(\lambda)$, then $\theta(1)=2(q_0-1)$. Exactly similar to the above discussion, we get a contradiction. Hence, $q_0=9$. Firstly, suppose that ${\rm cd}(T|\lambda)= \{4,8,10\}$. Now, consider $\varphi_1,\varphi_2\in {\rm Irr}(T|\lambda)$ such that $\varphi_1(1)=4$ and $\varphi_2(1)=10$. If either both $\varphi_1$ and $\varphi_2$ are extendible to $I_G(\lambda)$ or both $\varphi_1$ and $\varphi_2$ are not extendible to $I_G(\lambda)$,  then there exist $\psi_1, \psi_2\in{\rm Irr}(G|\lambda)$ such that ${\rm cod}(\psi_1)=(5/2){\rm cod}(\psi_2)$, a contradiction by Remark \ref{codpsl}. Hence, exactly one of the characters among  $\{\varphi_1,\varphi_2\}$ is extendible to $I_G(\lambda)$. If $\varphi_1$  is extendible to $I_G(\lambda)$, then there exist $\psi_1, \psi_2\in{\rm Irr}(G|\lambda)$ such that ${\rm cod}(\psi_1)=5{\rm cod}(\psi_2)$, a contradiction by Remark \ref{codpsl}. Also, if $\varphi_2$  is extendible to $I_G(\lambda)$, then there exist $\psi_1, \psi_2\in{\rm Irr}(G|\lambda)$ such that ${\rm cod}(\psi_1)=5/4{\rm cod} (\psi_2)$. By comparing the elements of ${\rm cod(PSL}(2,q))$, and the fact that $q$ is a power of 9, we get that ${\rm cod}(\psi_1)=q(q+1)/2$ and ${\rm cod}(\psi_2)=q(q-1)/2$. Hence, $(q+1)/(q-1)=(5/4)$, and so $q=9$, a contradiction as $q>q_0$.   
  
   Now, assume that ${\rm cd}(T|\lambda)= \{3,6,9,10\}$. Consider also $\varphi_1,\varphi_2\in {\rm Irr}(T|\lambda)$ such that $\varphi_1(1)=3$ and $\varphi_2(1)=9$. Similar to the above discussion, for exactly one element $i\in\{1,2\}$, $\varphi_i$ is extendible to $I_G(\lambda)$. If $\varphi_1$  is extendible to $I_G(\lambda)$, then there exist $\psi_1, \psi_2\in{\rm Irr}(G|\lambda)$ such that ${\rm cod}(\psi_1)=1.5{\rm cod}(\psi_2)$. Hence, by Remark \ref{codpsl}, we get that $q=3$ or $q=5$, a contradiction by the fact that $q$ is a power of $q_0=9$. Thus, $\varphi_2$  is extendible to $I_G(\lambda)$, and so there exist $\psi_1, \psi_2\in{\rm Irr}(G|\lambda)$ such that ${\rm cod}(\psi_1)=6{\rm cod}(\psi_2)$, a contradiction by Remark \ref{codpsl}. Therefore, we must have ${\rm cd}(I_G(\lambda)|\lambda)=\{6,12\}$. Let $\varphi_1,\varphi_2\in {\rm Irr}(T|\lambda)$ such that $\varphi_1(1)=6$ and $\varphi_2(1)=12$. If $12\in \cd(\Irr(I_G(\lambda)|\lambda))$, then by Step (3), $720|N|/12=60|N|\in \cod(G)$. Since $q$ is a power of $q_0=9$, $N$ must be a 3-group and $20\in\{(q-1),(q+1),(q-1)/2,(q+1)/2\}$, a contradiction. Therefore, the only possibility is that $\varphi_1$  is extendible to $I_G(\lambda)$ and $\varphi_2$ is not. As a result, there exist $\psi_1, \psi_2\in{\rm Irr}(G|\lambda)$ such that ${\rm cod}(\psi_1)=4{\rm cod}(\psi_2)$. By Remark \ref{codpsl}, the only possibility is $\epsilon(q)=-1$, $\cod(\psi_1)=q(q+1)$ and $\cod(\psi_2)=q(q-1)/2$. Hence, $q=3$, a contradiction by the fact that $q$ is a power of $q_0=9$.

 % $\theta(1)=2\varphi(1)$.  and , ${\rm cd}(I_G(\lambda)|\lambda)=$  or . Exactly similar to the previous case, we get that ${\rm cd}(I_G(\lambda)|\lambda)=\{12,24\}$, and so there exist $\psi_1, \psi_2\in{\rm Irr}(G|\lambda)$ such that ${\rm cod}(\psi_1)=2{\rm cod}(\psi_2)$. By Remark \ref{codpsl}, $60|N|=|{I}_G(\lambda)/N||N|/\theta(1)={\rm cod}(\psi_1)=q(q-\epsilon(q))$. Since $q$ is a power of 3, $N$ must be a 3-group, and so  $q-\epsilon(q)=20$, a contradiction. 

$\bullet$ Finally,  suppose that there exists $1_N\neq\lambda\in {\rm Irr}(N)$ such that ${ I}_G(\lambda)/N$ is isomorphic to the cyclic group of order $s$, where $s\in\{(q-1)/2,(q+1)/2\}$. By Step (4),  $|N|\in\{q-1,q+1\}$, and so $N$ is a 2-group. Therefore, there exists $1_N\neq\lambda_1\in {\rm Irr}(N)$ such that ${ I}_G(\lambda_1)/N$ contains a Sylow $2$-subgroup of $G/N$. By the structure of Sylow 2-subgroups of ${\rm PSL(2,q)}$, we get that ${ I}_G(\lambda_1)/N$ is isomorphic to a dihedral group, a contradiction as we explained.

%If $N$ is a $r$-group, then for all ${\theta_0}\in {\rm Irr}({ I}_G(\lambda_0)|\lambda_0)$, we have $q(q-1)|N|/\theta_0(1)\in {\rm cod}(G)$. Hence, $q=|N|=\theta_0(1)$, for every ${\theta_0}\in {\rm Irr}({ I}_G(\lambda_0)|\lambda_0)$. Using \cite[Lemma 5.2, Theorem 6.2]{isaacs}, we get that $|{ I}_G(\lambda_0)/N|=kq^2$, where $k$ is a natural number, a contradiction. Thus, $|N|\in\{q-1,q+1\}$, and so $N$ is a 2-group. Therefore, there exists $1_N\neq\lambda_1\in {\rm Irr}(N)$ such that ${ I}_G(\lambda_1)/N$ contains a Sylow $2$-subgroup of $G/N$. By the structure of Sylow 2-subgroups of ${\rm PSL(2,q)}$, we get that ${ I}_G(\lambda_1)/N$ is isomorphic to a dihedral group, a contradiction as we explained.

\bigskip

\textbf{Step (7).\,}  Final contradiction.

\bigskip
Let $q=r^f$, where $r$ is a prime. Using Steps (5) and (6), for every $1_N\neq\lambda\in{\rm Irr}(N)$, ${I}_G(\lambda)/N$ contains a Sylow $r$-subgroup of $G/N\cong {\rm PSL}(2,q)$. Hence, the pair $(G/N,N)$ satisfies $N_r$. Thus, Lemma \ref{size Nq} implies that $(|N|-1)/(|{\bf C}_N(R)|-1)=n_r(G/N)$, where $R\in {\rm Syl}_r(G/N)$. Let $|N|=p^n$ and $|{\bf C}_N(Q)|=p^m$, where $m$ and $n$ are natural numbers. If $n=s m$, then $(p^m)^{s-1}+(p^m)^{s-2}+\dots+1=q+1$.   Thus, $(p,s,m)=(r,2,f)$, and so $|N|=q^{2}$ which is a contradiction since by Step (3), $|N|\mid |G/N|$. Consequently, we get that $N=1$, as wanted. 
%$q=2^f$, then $(p,s,m)=(2,2,f)$, and if $q=r^\alpha$, where $r$ is an odd prime such that $r^\alpha>5$, then $(p,s,m)=(r,2,\alpha)$. Consequently, $|N|=q^{2}$, as wanted.
%\end{mainT}
%%%%%%%%%%%%%%%%%%%%%%%%%%%%%%%%%%%%%%%%%%%%%%%%%%%%%%%
%%%%%%%%%%%%%%%%%%%%%%%%%%%%%%%%%%%%%%%%%%%%%%%%%%%%%%
%%%%%%%%%%%%%%%%%%%%%%%%%%%%%%%%%%%%%%%%%%%%%%%%%%
%%%%%%%%%%%%%%%%%%%%%%%%%%%%%%%%%%%%%%%%%%%%%%%%%%%%%%%%%%%%%%%%%%%%%%%%%%%%%%%%%%%%%%%%%%%%%%%%%%%%%%%%%%%%%%%%%%%%%%%

%%%%%%%%%%%%%%%%%%%%%%%%%%%%%%%%%%%%%

\end{document}